\documentclass[12pt,reqno]{amsart}

\usepackage{amssymb}
\usepackage{graphicx}
\usepackage{amsfonts}
\usepackage{mathptmx}
\usepackage{dsfont}
\usepackage{tikz}
\usepackage{bm}
\usetikzlibrary{arrows}
\usepackage{mathtools}

\usepackage{rotating}
\usepackage{lscape}
\usepackage{hyperref}
\usepackage[mathscr]{euscript}
\usepackage{booktabs}

\usepackage[ruled,vlined,algosection,norelsize]{algorithm2e}

\usepackage{pgf, tikz}
\usepackage{color}

\definecolor{light}{gray}{0.65}

\numberwithin{equation}{section}

\def\rank{\operatorname{rank}}

\def\Iter{\operatorname{Iter}}

\def\lcm{\operatorname{lcm}}

\def\gp{\operatorname{gp}}

\DeclareMathOperator{\vol}{vol}

\DeclareMathOperator{\HB}{HB}

\def\ZZ{{\mathbb Z}}
\def\NN{{\mathbb N}}

\def\RR{{\mathbb R}}

\def\CC{{\mathbb C}}

\newtheorem{corollary}[algocf]{Corollary}
\newtheorem{theorem}[algocf]{Theorem}
\newtheorem{proposition}[algocf]{Proposition}

\theoremstyle{definition}
\newtheorem{definition}[algocf]{Definition}
\newtheorem{remark}[algocf]{Remark}

\newtheorem{example}[algocf]{Example}
\newtheorem{question}[algocf]{Question}

\numberwithin{algocf}{section}

\def\ttt#1{\texttt{#1}}

\usepackage{color}

\begin{document}

\title{On the consistency of score sheets of a round-robin football tournament}

\author{Bogdan Ichim}

\address{Bogdan Ichim \\ University of Bucharest \\ Faculty of Mathematics and Computer Science \\ Str. Academiei 14 \\ 010014 Bucharest \\ Romania \\ and \newline
 Simion Stoilow Institute of Mathematics of the Romanian Academy \\ Research Unit 5 \\ C.P. 1-764 \\ 010702 Bucharest \\ Romania}
\email{bogdan.ichim@fmi.unibuc.ro \\ bogdan.ichim@imar.ro}

\author{Julio Jos\'e Moyano-Fern\'andez}

\address{Universitat Jaume I, Campus de Riu Sec, Departamento de Matem\'aticas \& Institut Universitari de Matem\`atiques i Aplicacions de Castell\'o, 12071
Caste\-ll\'on de la Plana, Spain} \email{moyano@uji.es}

\thanks{
The second author was partially funded by MCIN/AEI/10.13039/501100011033 and by ``ERDF - A way of making Europe'', grant PGC2018-096446-B-C22, as well as by Universitat Jaume I, grant UJI-B2021-02.}

\begin{abstract} In this paper we introduce the submonoids $\mathscr{R}_n$, resp.~$\mathscr{C}_n$, of the monoid  $\mathscr{M}_n$ of ordered score sheets of a robin-round tournament played by $n$ teams for which the order is preserved after the leader team is disqualified, resp. all principal submatrices preserve the given ordering. We study (using both theoretical and  computational methods)
the most important invariants of these monoids, namely the Hilbert basis, the multiplicity, the Hilbert series and the Hilbert function. In particular we give a general description of the Hilbert basis of $\mathscr{R}_n$ and we show that $\mathscr{C}_n$ is Gorenstein for $n>2$.
%
\end{abstract}
\maketitle

\section{Introduction}

In the summer of 1992, Denmark national football team was informed about its sudden participation in  Euro'92 UEFA championship. Their surprising inclusion came after Yugoslavia, in a state of civil war, was not allowed to participate in the above tournament, according to the United Nations Security Council Resolution 757 (Implementing Trade Embargo on Yugoslavia). The Security Council, in \cite[paragraph 8(b)]{UNO}, decided that all States shall ``take the necessary steps to prevent the participation in sporting events on their territory of persons or groups representing the Federal Republic of Yugoslavia (Serbia and Montenegro)''. The Danes had missed out on qualification having finished as runners-up to Yugoslavia; they found suddenly themselves included in the eight-team tournament, for a detailed account see \cite{SCh}. It was even more surprising that they became champions of Euro'92.

This real-life example shows how external factors may affect a well-organized sport event. In the case under consideration, before Yugoslavia was kicked out, the Group 4 of the preliminary round consisted of five teams, which classified in the following order: Yugoslavia, Denmark, Northern Ireland, Austria and Faroe Islands. When Yugoslavia was not accepted in the next round, the Danes were called to replace them, since they were ordered immediately after Yugoslavia. Was this a fair decision?
In this particular situation, yes, it was, but this does not need to be always the case. Example \ref{ex:order_changed} shows that it is really possible that the disqualification of the top team will change the order among the remaining teams. If this is not the case, we say that a score sheet is \emph{runner-up consistent}. Moreover, we say that a score sheet is simply \emph{consistent}, if the order of the remaining teams is preserved, no matter which teams are eliminated from the competition.

Both runner-up consistent and consistent score sheets  may be endowed with the structure of an affine monoid and therefore they become naturally objects for an algebraic, combinatorial and statistical study.
In fact, they are submonoids of the affine monoid of \emph{ordered} score sheets, which was studied by the authors in \cite{IM}.
In this paper we continue this line of investigation,  by presenting a systematic study of both runner-up consistent and consistent score sheets of a round-robin football tournament played between $n$ teams.

We focus our attention on the most important invariants of these monoids.
More precisely, we study the Hilbert basis, the multiplicity, the Hilbert series and the Hilbert function for each monoid.
From a practical point of view, the most important invariants are the multiplicity and the Hilbert function, since they may be used for computing the probability of the events that runner-up consistent or consistent score sheets might appear as the result of a round-robin football tournament. However, it turns out that they are also the most challenging to compute.

Notice that combinatorics behind football and other sports have already caught some attention; for instance D. Zeilberger \cite{EZ} deduced a formula for the counting of the number of ways for the four teams of the group stage of the FIFA World Cup to each have $r$ goals \emph{for} and $r$ goals \emph{against} (this was the initial motivation for the authors study in \cite{IM}). Moreover, Zeilberger has recently given an interesting lecture about the number of possible histories in a football game \cite{Z2}. On the other hand, Kondratev, Ianovsky and Nesterov \cite{KIN} analyzed how scoring rules may be robust to adding or eliminating particular candidates.



This paper is structured as follows. In Section \ref{sec:pre} we briefly review some standard facts on rational cones and affine monoids. For details we refer the reader to Bruns and Gubeladze \cite{BG} or Bruns and Herzog \cite{BH}. The reader is supposed to be familiar with the basics of these topics. Section \ref{sec:scoresheets} provides a short exposition about ordered score sheets in the context of algebraic combinatorics; in particular, we introduce the submonoid $\mathscr{R}_n$ of the monoid of ordered score sheets of a round-robin tournament played by $n$ teams for which the given order is preserved after the leader team is disqualified, resp. its submonoid $\mathscr{C}_n$ for which all principal submatrices preserve the given ordering; we call them \emph{runner-up consistent}, resp. \emph{consistent} score sheets.

Section \ref{sec:main} contains the main results: we study the Hilbert basis of $\mathscr{R}_n$ for which a full description is given in Theorem \ref{theo:HB}. Moreover, Theorem \ref{theo:Gorenstein} shows that the monoid $\mathscr{C}_n$ is Gorenstein for every $n\geq 3$.

In Section \ref{sec:qua} we focus on particular questions about Hilbert series, Hilbert quasipolynomials and multiplicities. We present in detail the numerical data associated to the monoids $\mathscr{R}_3$ and $\mathscr{C}_3$, as well as partial data for the monoids $\mathscr{R}_n$ and $\mathscr{C}_n$ for some $n\geq 4$. The data is further used for applications. We finish this paper with a report on the computational experiments done using the software Normaliz \cite{Nmz} in Section \ref{sec:comp}.
\smallskip

The first author wishes to express his gratitude to the Department of Mathematics and the Institute of Mathematics and Applications of Castell\'o--IMAC in the Universitat Jaume I (Spain) for kind hospitality and support during his three-weeks-long research stay in Autumn 2019.


\section{Preliminaries}\label{sec:pre}

This paper deals with the theory of rational cones and affine monoids, as well as the theory of generating functions (Hilbert function, Hilbert series). Good references for these are the book of Bruns and Gubeladze \cite{BG} for the first issue, and the book of Bruns and Herzog for the second \cite{BH}. Based on \cite{BG}, Section 2 in our paper \cite{IM} gives also a short account on the essentials. We do not repeat all this introductory material here, but just a minimum in order to cover the essential elements in the theory and to fix notation.

\subsection{Affine monoids and cones}
A subset $C\subset \RR^d$ is called a \emph{rational cone} if it is the intersection of
finitely many closed linear rational halfspaces. The \emph{dimension} of a cone is the dimension of the smallest vector subspace of $\RR^d$ which contains it. If $\dim
C=d$, then the halfspaces in an irredundant representation of $C$  are uniquely determined.
A cone is \emph{pointed} if
$x,-x\in C$ implies $x=0$. In the following all cones $C$ are rational and pointed and we shall omit these attributes.

A hyperplane $H$ is called a \emph{supporting hyperplane} of a cone $C$ if $C\cap H\neq \emptyset$ and $C$ is contained in one of the closed halfspaces determined by $H$.
If $H$ is a supporting hyperplane of $C$, then $F=C\cap H$ is called a \emph{proper face} of $C$. It is convenient to consider also the empty set and $C$ as faces, the \emph{improper faces}.
The faces of a cone are themselves cones. A face $F$ with $\dim(F)=\dim(C)-1$ is called a \emph{facet}.
The faces of dimension $1$ of a pointed cone are called \emph{extreme rays} and the vectors with coprime integral components spanning them are called \emph{extreme integral generators}.
\medskip

An \emph{affine monoid} $M$ is a finitely generated (and isomorphic
to a) submonoid of a lattice $\ZZ^n$. It admits a unique minimal system of generators given by its irreducible elements, which is called the \emph{Hilbert basis} of $M$ and denoted by $\HB(M)$.
By $\gp(M)$ we denote the
subgroup generated by $M$, and by $r=\rank M$ its rank. The group $\gp(M)$ is isomorphic to $\ZZ^r$, and we may identify them. Let $C=\RR_+M\subset\RR^r$ be the cone generated by $M$.
Then we may write \[C=H^{+}_{\sigma_{1}}\cap \ldots
\cap H^{+}_{\sigma_{s}}\]
as an irredundant intersection of halfspaces defined by linear forms $\sigma_i$ on $\RR^r$, which are called \emph{support
forms} of $M$, after they have been further specialized such that
$\sigma_i(\ZZ^r)=\ZZ$. The last condition
amounts to the requirement that the $\sigma_i$ have coprime
integral coefficients. (Such linear forms are called
\emph{primitive}.) Note that after this standardization, we can assume that $\sigma_i$ is the $\ZZ^r$-height above $H_{\sigma_{i}}$. We also set
$\mathrm{int}(M)=M\cap\mathrm{int}(\RR_+M)$.
Further, we define the
\emph{standard map}
$$
\sigma:M\to\ZZ_+^s,\qquad
\sigma(x)=\bigl(\sigma_1(x),\dots,\sigma_s(x)\bigr).
$$

In this paper the Gorenstein condition for monoids will play an important role. For normal affine monoids we have the following characterization (cf. \cite[Theorem 6.33]{BG}):

\begin{theorem}\label{theo:Bruns_Gorenstein} Let $M$ be a normal affine monoid. The following are equivalent:
\begin{enumerate}
\item $M$ is Gorenstein;
\item there exists $x\in \mathrm{int}(M)$ such that $\mathrm{int}(M)=x+M$;
\item there exists $x\in M$ such that $\sigma_i(x)=1$ for all support forms $\sigma_i$ of $\mathbb{R}_+M$.
\end{enumerate}
\end{theorem}

\subsection{Hilbert series and functions of affine monoids} Let $C\subset \RR^d$ be a full dimensional cone, i.e. $\dim C=d$, and let $M=C\cap\ZZ^d$ be the corresponding affine monoid.
A \emph{$\ZZ$-grading} of  $\ZZ^d$ is a linear map $\deg: \ZZ^d\to\ZZ$. If all sets $\{x\in M: \deg x=u\}$ are finite, then the \emph{Hilbert series} of $M$
with respect to the grading $\deg$ is defined to be the formal Laurent series
$$
H_M(t)=\sum_{i\in \ZZ} \#\{x\in M: \deg x=i\}t^{i}=\sum_{x\in M}t^{\deg x}.
$$
We shall assume in the following that $\deg x >0$ for all nonzero $x\in M$ (such a grading does exist given the fact that $C$ is pointed)
and that there exists an $x\in\ZZ^d$ such that $\deg x=1$.
Then the Hilbert series $H_M(t)$ is the Laurent expansion of a rational function at the origin:

\begin{theorem}[Hilbert, Serre, Ehrhart, Stanley]\label{theo:HS} Let $M=C\cap\ZZ^d$ be as above.
	The Hilbert series of $M$ may be written in the form
	$$
	H_M(t)=\frac{R(t)}{(1-t^e)^d}=\frac{h_0+\cdots+h_ut^u}{(1-t^e)^d},\qquad R(t)\in\ZZ[t], 
	$$
	where $e$ is the $\lcm$ of the degrees of the extreme integral generators of $C$ and $h_u\neq 0$.
\end{theorem}

Note that the above representation is not unique. This problem is discussed in Bruns, Ichim and S\"oger \cite[Section 4]{BIS}.
\medskip

An  equivalent statement can be given using the \emph{Hilbert function}
$$
H(M,i)=\#\{x\in M: \deg x=i\},
$$
as considered e.g. by Bruns and the first author in \cite{bipams}, namely

\begin{theorem}[Hilbert, Serre, Ehrhart, Stanley]\label{theo:HP}
There exists a quasipolynomial $Q$ with rational coefficients, degree $d-1$ and period $p$ dividing $e$ such that $H(M,i)=Q(i)$ for all $i\ge0$.
\end{theorem}

The cone $C$ together with the grading $\deg$ define the rational
polytope
\[
\mathscr{P}=\mathscr{P}_{C}=C\cap \{x\in \RR^d:\deg x=1\}.
\]

It is not hard to show that the $p$ polynomial components of $Q$ have the same degree $d-1$ and the same leading coefficient
$a_{d-1}=\frac{\vol(\mathscr{P})}{(d-1)!},$
where $\vol(\mathscr{P})$ is the lattice normalized volume of $\mathscr{P}$ (a lattice simplex of smallest possible volume has volume $1$).
The parameter $e(M)=\vol(\mathscr{P})=a_{d-1}(d-1)!$ is called the \emph{multiplicity} of $M$.

\section{Consistency of ordered score sheets}\label{sec:scoresheets}

In this section first we briefly recall the terminology on score sheets of round-robin football tournaments that was introduced in \cite{IM}.

A \emph{score sheet} $S$ of a round-robin football tournament played between $n$ teams $T_1,\dots,T_n$ is a table of the form
$$
\begin{array}{c|c c c c}
&T_1&T_2& \cdots &T_n\\
\hline
T_1&\ast&g_{12}&\ldots&g_{1n}\\
T_2&g_{21}&\ast&\ldots&g_{2n}\\
\vdots&\vdots&&\ddots&\vdots\\
T_n&g_{n1}&g_{n2}&\ldots&\ast\\
\end{array}
$$
where $g_{ij} \in \mathbb{Z}_{+}$ and $g_{ij}$ represents the number of goals that the team $T_i$ scored the team $T_j$.
We denote by $\mathscr{S}_n$ the set of all score sheets of a round-robin football tournament played between $n$ teams.
Observe that an entry $g_{ij}$ may be any nonnegative integer.
\medskip

Further, we denote by $g_i$ the total number of goals scored by $T_i$, i.e.
$$
g_i=\sum_{\substack{j=1\\i\neq j}}^{n} g_{ij}.
$$
We also set $g_S=(g_1, \ldots, g_n)$, and the $\ZZ$-grading $\deg S=|\!|S|\!|=g_1+\cdots + g_n$. Then
$$|\!|S_1+S_2|\!|=|\!|S_1|\!|+|\!|S_2|\!|.$$
\medskip

We say that a score sheet $S\in \mathscr{S}_n$ of a round-robin football tournament played between $n$ teams is an \emph{ordered score sheet} if it satisfies the following inequalities:
$$
\sum_{\substack{j=1\\i\neq j}}^{n} g_{i,j} \geq \sum_{\substack{j=1\\i+1\neq j}}^{n} g_{i+1,j} \ \ \ \forall i=1,\dots,n-1,
$$
or, equivalently, the inequalities $g_1\ge g_2 \ge \cdots\ge g_n$. We remark that this is in fact the form in which most score sheets are presented.
It is easy to see that the set of the ordered score sheets of a round-robin football tournament played between $n$ teams is a submonoid of $\mathscr{S}_n$,
which we denote in the following by  $\mathscr{M}_n$.
\medskip

Next, we introduce the new terminology that will be used in the following. An ordered score sheet is called \emph{consistent} if all its principal submatrices are ordered score sheets.
In other words, the removal of all possible sets of rows and columns such that the indices of the deleted rows are the same as the indices of the deleted columns preserves the original order of the participants.
An ordered score sheet is called \emph{runner-up consistent}  if the score sheet resulting from the deletion of the first row and the first column preserves the ordering of the remaining tournament participants. These definitions make sense if $n\ge3$, which will be a general assumption throughout this paper. The set of all consistent, resp. runner-up consistent, score sheets of $n$ teams is a submonoid of  $\mathscr{M}_n$ denoted by $\mathscr{C}_n$, resp. $\mathscr{R}_n$. We have the inclusions
$$
\mathscr{C}_n \subseteq  \mathscr{R}_n \subseteq  \mathscr{M}_n \subseteq  \mathscr{S}_n.
$$

\begin{example}\label{ex:order_changed}
Consider the score sheet $S$ given by
\small
$$
\begin{array}{c|c c c c c || c}
&\mathrm{T_1}&\mathrm{T_2}& \mathrm{T_3} &\mathrm{T_4}&\mathrm{T_5}&\sum\\
\hline
\mathrm{T_1}&\ast&0&2&2 &2&6\\
\mathrm{T_2}&2&\ast&1&1&1&5\\
\mathrm{T_3}&0&1&\ast&2&1&4\\
\mathrm{T_4}&0&1&0&\ast&2&3\\
\mathrm{T_5}&0&1&1&0&\ast&2
\end{array}
$$
\normalsize
to which we have amended an extra column in order to indicate the total number of scored goals rowwise. By looking at the submatrix
\small
$$
\begin{array}{c| c c c c || c}
&\mathrm{T_2}& \mathrm{T_3} &\mathrm{T_4}&\mathrm{T_5}&\sum\\
\hline
\mathrm{T_2}&\ast&1&1&1&3\\
\mathrm{T_3}&1&\ast&2&1&\mathbf{4}\\
\mathrm{T_4}&1&0&\ast&2&3\\
\mathrm{T_5}&1&1&0&\ast&2
\end{array}
$$
\normalsize
(which was obtained by the deletion of the first row and the first column) it is clear that $S \notin \mathscr{R}_n$, therefore $S\notin \mathscr{C}_n$ neither.

\end{example}

\section{Main results}\label{sec:main}

\subsection{The Hilbert basis of the monoid $\mathscr{R}_n$}

In the first part of this section we present a complete description of the Hilbert basis of the monoid $\mathscr{R}_n$ of ordered score sheets of a round-robin football tournament played between $n$ teams which preserves the ordering of the remaining tournament participants after the leader team has been withdrawn from the tournament (as in the case of Yugoslavia in 1992).

\begin{theorem}\label{theo:HB}
The Hilbert basis of the monoid $\mathscr{R}_n$ is given by the the union of the sets $A_n$ and $B_n$ of ordered score sheets with the following properties:
\begin{enumerate}
\item[(A)] There exists $r\in \ZZ_+$, with $1\le r \le n$ such that
\begin{enumerate}
\item[(1)] The teams $T_1, \ldots , T_r$ have scored exactly one goal against any other team different from $T_1$;
\item[(2)] The teams $T_{r+1}, \ldots, T_n$ have scored no goal against other team.
\end{enumerate}
The set of these score sheets is denoted by $A_n$.
\item[(B)] There exist $r,q \in \ZZ_+$, with $1\le r < q \le n$ such that
\begin{enumerate}
\item[(1)] The teams $T_1, \ldots , T_r$ have scored exactly one goal against any other team different from $T_1$;
\item[(2)] The teams $T_{r+1}, \ldots, T_q$ have scored precisely one goal against $T_1$;
\item[(3)] The teams $T_{q+1}, \ldots , T_n$ have scored no goal against any other team.
\end{enumerate}
The set of these score sheets is denoted by $B_n$.
\end{enumerate}
The Hilbert basis of $\mathscr{R}_n$ equals the set of extreme integral generators of the cone $C_n=\RR_+\mathscr{R}_n$
and the polytope with vertices the Hilbert basis elements is compressed.
Moreover, for the number of elements in $A_n$ and $B_n$ we obtain that
\[
\# A_n=(n-1)\sum_{i=0}^{n-1} (n-2)^{i}\ \ \text{and}\ \ \# B_n=(n-1)\sum_{i=0}^{n-1} (n-2)^{i}(n-i-1).
\]
Finally, the number of elements in the Hilbert basis of the monoid $\mathscr{R}_n$ is given by
\[
\#\HB(\mathscr{R}_n)=\# (A_n\cup B_n) =(n-1)\sum_{i=0}^{n-1}(n-2)^i(n-i)=(n-1)\sum_{k=0}^{n-1}\sum_{i=0}^k (n-2)^{i}.
\]
\end{theorem}

\begin{proof}
First we show that $A_n\cup B_n$ is a set of generators of $\mathscr{R}_n$, i.e. we have to prove that each element in $\mathscr{R}_n$ can be written as a linear combination with positive coefficients of elements of $A_n\cup B_n$.
\medskip

Let $S \in \mathscr{R}_n$ and consider the vectors $g_S$ and
$$
\bar{g}_S=(\bar{g}_1,\bar{g}_2,\bar{g}_3,\ldots,\bar{g}_n)=g_S-(0,g_{21},g_{31},\ldots,g_{n1}).
$$
(The vectors $\bar{g}_S$ are reflecting the situation after the disqualification of initial top team.)
From the definition of $S \in \mathscr{R}_n$ we get the inequalities
$$
g_1\ge g_2 \ge g_3\ge \cdots\ge g_n \ \text{and} \ \bar{g}_2 \ge \bar{g}_3\ge \cdots\ge \bar{g}_n.
$$
Here two possible situations arise:
\begin{enumerate}
\item[\textbf{(1)}] There exists $i$ such that $g_i>\bar{g}_i$;
\item[\textbf{(2)}] $g_S=\bar{g}_S$.
\end{enumerate}
In the following we prove that situation \textbf{(1)} can be reduced to situation \textbf{(2)}.

\noindent \textbf{(1)} Assume the existence of an index $i$ such that $g_i>\bar{g}_i$. We show that $S=K+\sum_{i} H_i$, where $H_i\in B_n$ and $g_K=\bar{g}_K$. Let $q=\mathrm{max}\{i|g_i>\bar{g}_i \}$. Then either $q=n$, or $q<n$ and $g_q>\bar{g}_q\ge\bar{g}_{q+1}= g_{q+1}$. So
we have that either
$$
g_1\ge g_2 \ge g_3\ge \cdots\ge g_{q=n},
$$
or
$$
g_1\ge g_2 \ge g_3\ge \cdots\ge g_q>g_{q+1}\ge g_{q+2}\ge \cdots \ge g_n.
$$

Since $g_1=\bar{g}_1$, there exist always $r<q$ such that
$g_r=\bar{g}_r$ and $g_i>\bar{g}_i$ for $i=r+1,\ldots,q$. Further $\bar{g}_r=g_r\ge g_{r+1} > \bar{g}_{r+1}$, so we have that
$$
\bar{g}_1=g_1\ge g_2\ge\bar{g}_2 \ge \bar{g}_3\ge \cdots\ge \bar{g}_r>\bar{g}_{r+1}\ge\bar{g}_{r+2}\ge \cdots \ge \bar{g}_n.
$$

We may assume $\bar{g}_1\ge\bar{g}_2 \ge  \cdots\ge \bar{g}_r > 0$; this means that for $i=1,2, \ldots, r$ we may associate a $j_i> 1$ such that $g_{i,j_i} \neq 0$. We consider $H \in B_n$ defined by
$$
H_{ab}=\left \{
\begin{array}{lll}
 1 & \text{if}\ \{a,b\} = \{i,j_i\}\ \text{for}\ i=1,\ldots,r, \\
 1 & \text{if}\ \{a,b\} = \{i,1\}\ \text{for}\ i=r+1,\ldots,q, \\
 0 & \text{else.}
\end{array}
\right.
$$

\noindent Then $g_{S-H}=(g_1-1,\ldots,g_q-1, g_{q+1},\ldots,g_n)$ and $\bar{g}_{S-H}=(\bar{g}_1-1,\ldots,\bar{g}_r-1, \bar{g}_{r+1},\ldots,\bar{g}_n)$.
It follows that they satisfy the inequalities from the definition of $S \in \mathscr{R}_n$,
so we get that $S-H \in \mathscr{R}_n$. This procedure can be repeated with $S-H$ in place of $S$ until all elements on the first column are $0$.

\noindent \textbf{(2)} Now assume $g_S=\bar{g}_S$. The following argument is similar to the one given in part \textbf{(1)} of the proof of \cite[Theorem 4.1]{IM}. Observe that $A_n$ is the subset of $\mathscr{R}_n$ containing all score sheets $P\in \mathscr{R}_n$ such that $g_P=\bar{g}_P$ and
the vector $g_P$ is one of the following:
\begin{align*}
p_1&=(1,0,0,\ldots , 0);\\
p_2&=(1,1,0,\ldots , 0);\\
p_3&=(1,1,1,\ldots , 0);\\
\dots&\ldots \ldots \ldots \ldots \ldots \dots  \\
p_n&=(1,1,1,\ldots , 1).
\end{align*} Consider $S\in \mathscr{R}_n$ such that $g_S=\bar{g}_S$.  Then there is a unique writing
\[
g_S=\sum_{i=1}^{n} a_ip_i \ \ \mbox{~with~\ } a_i \in \mathbb{Z}_{+}.
\]
We may assume that $a_r \neq 0$; this means that for $i=1, \ldots, r$ we may associate $j_i>1$ such that $g_{i,j_i} \neq 0$. We consider $H \in A_n$ which has precisely $1$ at the position $(i,j_i)$, and $0$ otherwise. Then $S-H \in \mathscr{R}_n$. Since $|\!|S-H|\!| < |\!|S|\!|$, it follows by induction that $S=\sum_{i} H_i$ with $H_i\in A_n$.
\medskip

We have shown that $A_n\cup B_n$ is a set of generators of $\mathscr{R}_n$. From this point on, one can directly show that the generators $A_n\cup B_n$ form a Hilbert basis of the monoid $\mathscr{R}_n$
as it is done in proof of \cite[Theorem 4.1]{IM}. (It is enough to show that for every distinct $A, B\in A_n\cup B_n$ the difference $A-B \not\in\mathscr{R}_n$, then apply \cite[Corollary 2.3]{IM}.)
However, in the following we present a more conceptual proof that
is inspired by a remark of Winfried Bruns on \cite{IM}. (Notice that for the following argument it is enough to show that $A_n\cup B_n$ generate the cone $C_n=\RR_+\mathscr{R}_n$.)
\medskip

Apart from the $\ZZ$-grading given by $\deg S=|\!|S|\!|=g_1+\cdots + g_n$ considered in Section \ref{sec:scoresheets}, there is another interesting $\ZZ$-grading, namely
$$\deg_1 S=|\!|S|\!|_1=g_1=\bar{g}_1.$$
The cone $C_n=\RR_+\mathscr{R}_n$, together with the grading $\deg_1$, define the rational
polytope
\[
\mathscr{P}_n=C_n\cap \{x\in \RR^{n(n-1)}:\deg_1 x=1\}.
\]

We have $C_n=\RR_+(A_n\cup B_n)$, which means that the set of extreme integral generators of the cone $C_n$ is a subset of $A_n\cup B_n$.
Remark that for all $S\in A_n\cup B_n$ it holds that $\deg_1 S=1$. Then $A_n\cup B_n \subset \mathscr{P}_n$ and since it contains
the extreme integral generators of the cone $C_n$, it follows that the set of vertices $V_n$ of $\mathscr{P}_n$ is a subset of $A_n\cup B_n$. Moreover, $\mathscr{P}_n$
is a $(0,1)$-integral polytope.
In other words, we conclude that the monoid $\mathscr{R}_n$ is the Ehrhart
monoid of the $(0,1)$-integral polytope $\mathscr{P}_n$.
\medskip

Further, $\mathscr{P}_n$ is the solution of the system of linear inequalities
\begin{align*}
g_1=1;& \\
0\le g_{i}-g_{i+1} \le 1& \ \ \ \forall \ i=1,\dots,n;\\
0\le \bar{g}_{i}-\bar{g}_{i+1} \le 1& \ \ \ \forall \ i=2,\dots,n;\\
0\le g_{ij} \le 1& \ \ \ \forall \ 1\le i\not=j\le n.
\end{align*}

It follows directly from Ohsugi and Hibi \cite[Theorem 1.1]{OH} that the polytope $\mathscr{P}_n$ is \emph{compressed}.
(The class of compressed polytopes was introduced by Stanley in \cite{S1}).
In fact, $\mathscr{P}_n$ has \emph{width one} with respect to all its facets (i.e. it lies between the hyperplane spanned by
one facet and the next parallel lattice hyperplane, for all its facets) and such a polytope is compressed by an unpublished
result due to Francisco Santos (see \cite[Section 8.3 Compressed Polytopes]{HNP} for a detailed discussion).
\medskip

In particular, all lattice points in $\mathscr{P}_n$ are vertices (see \cite[Definition 8.9]{HNP}), which means that $V_n=A_n\cup B_n$, i.e.
$A_n\cup B_n$ equals the set of extreme integral generators of the cone $C_n$.
Since the set of extreme integral generators of the cone $C_n$ is included in the Hilbert basis of the monoid $\mathscr{R}_n$,
which in turn is included in any set of generators of $\mathscr{R}_n$, we conclude that Hilbert basis of the monoid $\mathscr{R}_n$ equals the set $A_n\cup B_n$.
\end{proof}

\begin{corollary}\label{prop:compressed}
Let $\mathscr{P}_n$ be defined as above. Then every pulling triangulation of $\mathscr{P}_n$ is unimodular. Moreover, since $\mathscr{P}_n$ contains the Hilbert basis of $\mathscr{R}_n$, $\mathscr{P}_n$ is integrally closed.
\end{corollary}

Note that, on the one hand, a compressed polytope is not automatically a \emph{totally unimodular} polytope (i.e. a polytope with all triangulations unimodular).
On the other hand, it is true that a totally unimodular polytope is compressed. We will come back to this in Remark \ref{rem:tri}.

\begin{example}
The Hilbert basis of the monoid $\mathscr{R}_3$ consists of 12 elements, namely
\small
$$
\begin{array}{cccc}
\begin{array}{c|c c c}
&T_1&T_2 &T_3\\
\hline
T_1&\ast&0&1\\
T_2&0&\ast &0\\
T_3&0&0&\ast
\end{array}
&
\begin{array}{c|c c c}
&T_1&T_2 &T_3\\
\hline
T_1&\ast&1&0\\
T_2&0&\ast &0\\
T_3&0&0&\ast
\end{array}
&
\begin{array}{c|c c c}
&T_1&T_2 &T_3\\
\hline
T_1&\ast&0&1\\
T_2&0&\ast &1\\
T_3&0&0&\ast
\end{array}
&
\begin{array}{c|c c c}
&T_1&T_2 &T_3\\
\hline
T_1&\ast&0&1\\
T_2&1&\ast &0\\
T_3&0&0&\ast
\end{array}
\end{array}
$$
$$
\begin{array}{cccc}
\begin{array}{c|c c c}
&T_1&T_2 &T_3\\
\hline
T_1&\ast&1&0\\
T_2&0&\ast &1\\
T_3&0&0&\ast
\end{array}
&
\begin{array}{c|c c c}
&T_1&T_2 &T_3\\
\hline
T_1&\ast&1&0\\
T_2&1&\ast &0\\
T_3&0&0&\ast
\end{array}
&
\begin{array}{c|c c c}
&T_1&T_2 &T_3\\
\hline
T_1&\ast&0&1\\
T_2&0&\ast &1\\
T_3&0&1&\ast
\end{array}
&
\begin{array}{c|c c c}
&T_1&T_2 &T_3\\
\hline
T_1&\ast&0&1\\
T_2&0&\ast &1\\
T_3&1&0&\ast
\end{array}
\end{array}
$$
$$
\begin{array}{cccc}
\begin{array}{c|c c c}
&T_1&T_2 &T_3\\
\hline
T_1&\ast&0&1\\
T_2&1&\ast &0\\
T_3&1&0&\ast
\end{array}
&
\begin{array}{c|c c c}
&T_1&T_2 &T_3\\
\hline
T_1&\ast&1&0\\
T_2&0&\ast &1\\
T_3&0&1&\ast
\end{array}
&
\begin{array}{c|c c c}
&T_1&T_2 &T_3\\
\hline
T_1&\ast&1&0\\
T_2&0&\ast &1\\
T_3&1&0&\ast
\end{array}
&
\begin{array}{c|c c c}
&T_1&T_2 &T_3\\
\hline
T_1&\ast&1&0\\
T_2&1&\ast &0\\
T_3&1&0&\ast
\end{array}
\end{array}
$$
\normalsize
\end{example}
\bigskip

\subsection{Gorensteinness of the monoid $\mathscr{C}_n$} In the second part of this section we prove the following remarkable fact.

\begin{theorem}\label{theo:Gorenstein} The monoid $\mathscr{C}_n$ is Gorenstein.
\end{theorem}

\begin{proof} First we want to write the cone $C=\RR_+ \mathscr{C}_n\subset\RR^{n(n-1)}$ (which is generated by $\mathscr{C}_n$) in the form
\[
C=H^{+}_{\sigma_{1}}\cap \ldots\cap H^{+}_{\sigma_{s}},
\tag{$\ast$}
\]
i.e.~as an  intersection of halfspaces defined by linear forms on $\RR^{n(n-1)}$ (not necessary irredundant).

In order to make our proof easier to understand  we may write in the following both the elements $x$ of $\RR^{n(n-1)}$ and the linear forms $\sigma$ on $\RR^{n(n-1)}$ in the form
$$
\begin{array}{cc}
x=
\begin{array}{c|c c c c}
&1&2& \cdots &n\\
\hline
1&\ast&x_{12}&\ldots&x_{1n}\\
2&x_{21}&\ast&\ldots&x_{2n}\\
\vdots&\vdots&&\ddots&\vdots\\
n&x_{n1}&x_{n2}&\ldots&\ast\\
\end{array}
&
\sigma=
\begin{array}{c|c c c c}
& 1&2& \cdots &n\\
\hline
1&\ast&\sigma_{12}&\ldots&\sigma_{1n}\\
2&\sigma_{21}&\ast&\ldots&\sigma_{2n}\\
\vdots&\vdots&&\ddots&\vdots\\
n&\sigma_{n1}&\sigma_{n2}&\ldots&\ast\\
\end{array}
\end{array}
$$
Then
$$
\sigma(x)=\sum_{\substack{a,b= 1\\ \\ a\neq b}}^{a,b= n} \sigma_{ab} x_{ab}.
$$
Let $\delta^{ij}$ be the linear form on $\RR^{n(n-1)}$ defined by
$$
\delta_{ab}^{ij}=\left \{
\begin{array}{ll}
1 & \text{if}\ \{a,b\} = \{i,j\}, \\
0 & \text{else.}
\end{array}
\right.
$$
The inequality $\delta^{ij}(x)\ge 0$ is equivalent to the condition $x_{ij}\ge 0$.

Let $\mathcal{P}([n])$ be the power set of the set $[n]=\{1,\ldots,n\}$. Then we can write
$$
C=\Big[\bigcap_{\substack{i,j\in [n]\\ \\ i< j}} H^{+}_{\delta^{ij}} \Big]\cap\Big[\bigcap_{\substack{i,j\in [n]\\ \\ i> j}} H^{+}_{\delta^{ij}} \Big]\cap\Big[\bigcap_{\substack{P\in \mathcal{P}([n])\\ \\\#P\ge 2}} (\bigcap_{k=1}^{k<\#P}H^{+}_{\sigma_{k}^P})\Big]
$$
where, for each $P$, the linear forms $\sigma_{k}^P$ on $\RR^{n(n-1)}$ correspond to the conditions that endowed the principal submatrix induced by $P$ with the structure of ordered score sheet. (By principal submatrix induced by $P$ we understand the submatrix formed by selecting, for each $i\in P$, the $i$-th row and column.) More precisely, let $\Iter(k,P)$ be the $k$-th element of $P$. Then the linear form
$\sigma_{k}^P$ is defined by
$$
(\sigma_{k}^P)_{ab}=\left \{
\begin{array}{lll}
 1 & \text{if}\ a=\Iter(k,P)\ \text{and}\ b\in P\setminus \{\Iter(k,P)\}, \\
-1 & \text{if}\ a=\Iter(k+1,P)\ \text{and}\ b\in P\setminus \{\Iter(k+1,P)\}, \\
 0 & \text{else.}
\end{array}
\right.
$$

Now assume $\#P=2$, that is $P=\{i,j\}$. We may set $i<j$. There is only one linear form
$\sigma_{1}^P=\sigma^{ij}$ corresponding to the only condition induced by $P$, and this form on $\RR^{n(n-1)}$ is defined by
$$
\sigma_{ab}^{ij}=\left \{
\begin{array}{lll}
 1 & \text{if}\ \{a,b\} = \{i,j\}, \\
-1 & \text{if}\ \{a,b\} = \{j,i\}, \\
 0 & \text{else.}
\end{array}
\right.
$$
The inequality  $\sigma^{ij}(x)\ge 0$ is equivalent to the inequality $x_{ij}-x_{ji}\ge 0$. Since we also have the inequality $x_{ji}\ge 0$, we
deduce that the inequality  $x_{ij}\ge 0$ is redundant. In other words, the halfspaces $H^{+}_{\delta^{ij}}$ from the writing of $C$ are redundant if $i<j$.
We  conclude that
\[
C=\Big[\bigcap_{\substack{i,j\in [n]\\ \\ i> j}} H^{+}_{\delta^{ij}} \Big]\cap\Big[\bigcap_{\substack{P\in \mathcal{P}([n])\\ \\\#P\ge 2}} (\bigcap_{k=1}^{k<\#P}H^{+}_{\sigma_{k}^P})\Big].\tag{$\ast \ast$}
\]

For $t\in \NN$ with $t\ge2$, let $e_t\in\RR^{t(t-1)}$ be the element
$$
\begin{array}{c|c c c c c}
&1&2& \cdots & t-1 & t\\
\hline
1&\ast&2&\ldots& 2 & 2 \\
2& 1 &\ast&\ldots& 2 & 2\\
\vdots&\vdots&&\ddots&\vdots&\vdots\\
t-1& 1 & 1 &\ldots &\ast & 2\\
t& 1 & 1 &\ldots & 1 &\ast.\\
\end{array}
$$

First, it is clear that for $i>j$ we have that $\delta^{ij}(e_n)=1$.
\medskip

Next, observe that for any $P\in\mathcal{P}([n])$ with $\#P=t$, we have that the principal submatrix induced by $P$ on $e_n$
is exactly $e_t$.
\medskip

Finally, if we denote by $e_{t}^k$ the sum of elements of $e_t$ on the $k$-row for $t\in\NN$, $t\ge 2$ and for all $1\leq k \leq t$,
it is easily seen that
$$
e_{t}^k-e_{t}^{k+1}=1\ \text{for all}\ 1\leq k \leq t-1,
$$
since any two consecutive rows (or columns) of $e_t$ differ by precisely one element. In other words, we have that
$$
\sigma_{k}^P(e_n)=1,
$$
for all $P$ and $k$.
\medskip

The representation $(**)$ is a representation of type $(*)$  as an intersection of halfspaces (not necessary irredundant) and moreover there exists $e_n \in \mathscr{C}_n$
such that $\sigma(e_n)=1$ for all $\sigma$ in the writing. (It is clear that in any such representation a halfspace can be omitted if it contains the intersection of the remaining halfspaces, thus obtaining an irredundant representation.) Then Theorem \ref{theo:Bruns_Gorenstein} can be applied to obtain the conclusion.
\end{proof}

\begin{remark} In fact, the representation $(**)$ is irredundant.
Assume that representation $(**)$ is not irredundant. This implies that there exists a $\sigma$ which is not an extreme integral generator of the dual cone,
which is generated in $\RR^{n(n-1)}$ by all $\delta^{ij}$ with $i>j$ and all $\sigma_{k}^P$.
Then it may be written as a linear combination  of the extreme integral generators with nonnegative \emph{rational} coefficients:
$$
\sigma = \sum_{q} \alpha_q\delta^{ij}_q + \sum_{r} \beta_r\sigma_{kr}^P.
$$
It is immediately clear that only the trivial writing is possible for $\delta^{ij}$, so we can assume that $\sigma$ is of type $\sigma_{k}^P$.
Then $|\!|\sigma|\!|=0$. Since $|\!|\sigma_{kr}^P|\!|=0$ for all $r$ and $\alpha_q\ge 0$ for all $q$ we deduce that $\alpha_q=0$ for all $q$. It follows that
$$
\sigma = \sum_{r} \beta_r\sigma_{kr}^P.
$$
Note that for each $\sigma_{k}^P$ there exists a unique pair $(a,b)$ such that $(\sigma_{k}^P)_{ab}=1$ and $(\sigma_{k}^P)_{ba}=-1$, which we call \emph{signature}.
(In fact, this is only possible if $a=\Iter(k,P)$ and $b=\Iter(k+1,P)$.) For any other pair $(a,b)$ with $a,b\in [n]$ it is true that either $(\sigma_{k}^P)_{ab}=0$ or $(\sigma_{k}^P)_{ba}=0$.
 Now we can deduce that all summands  must have the same signature as $\sigma$. (Otherwise, assuming a summand in the linear combination has a different signature $(a',b')$, $\sigma$ would have nonzero entries $(\sigma_{k}^P)_{a'b'}$ and $(\sigma_{k}^P)_{b'a'}$.) By looking at the sum in the entries corresponding to the signature we deduce that  $\sum_{r} \beta_r=1$. Finally, if we look at the sum in any nonzero entry of $\sigma$, we see that each $\sigma_{kr}^P$ has a nonzero entry at the same place, so $\sigma_{kr}^P = \sigma$.
\end{remark}

The Gorensteinness of the monoid $\mathscr{C}_n$ has several interesting consequences.

\begin{corollary}\label{cor:Gorenstein}
Let $e_n$ be as in the proof of Theorem \ref{theo:Gorenstein} and set $v=\deg e_n=\frac{3n(n-1)}{2}$. Then:
\begin{enumerate}
\item $\mathrm{int}(\mathscr{C}_n)=e_n+\mathscr{C}_n$;
\item Consider the presentation of $H_{\mathscr{C}_n}(t)$ as in Theorem \ref{theo:HS}. Then the coefficients of $R(t)$ satisfy the relations $h_i=h_{u-i}$ for   $i=0,\ldots,u$ (in this case we say that the $h$-vector is palindromic);
\item Consider the presentation of $H(\mathscr{C}_n,t)$ as in Theorem \ref{theo:HP}, then
$$
Q(-k)=-Q(k-v)\ \ \text{for all}\ \ k\in \ZZ.
$$
\end{enumerate}
\end{corollary}

\begin{proof} Statement (1) follows from Theorem \ref{theo:Bruns_Gorenstein}. Assertion (2) is deduced from \cite[Corollary 6.43 (b)]{BG} with the remark that all multiplications with
$1+t^{e_i}+\cdots+t^{e-e_i}$ (where $e_i$ is a divisor of $e$) preserve the palindromic property. Finally, (3) follows from (1) by applying \cite[Corollary 6.42]{BG}, \cite[Corollary 6.43 (a)]{BG} and \cite[Proposition 6.47 (b)]{BG}.
\end{proof}

We end this section making some remarks on the Hilbert basis of the monoid $\mathscr{C}_n$.

\begin{example}\label{ex:HBC3}
The Hilbert basis of the monoid $\mathscr{C}_3$ consists of 10 elements, namely:
\small
$$
\begin{array}{cccc}
\begin{array}{c|c c c}
&T_1&T_2 &T_3\\
\hline
T_1&\ast&0&1\\
T_2&0&\ast &0\\
T_3&0&0&\ast
\end{array}
&
\begin{array}{c|c c c}
&T_1&T_2 &T_3\\
\hline
T_1&\ast&1&0\\
T_2&0&\ast &0\\
T_3&0&0&\ast
\end{array}
&
\begin{array}{c|c c c}
&T_1&T_2 &T_3\\
\hline
T_1&\ast&0&1\\
T_2&0&\ast &1\\
T_3&0&0&\ast
\end{array}
&
\begin{array}{c|c c c}
&T_1&T_2 &T_3\\
\hline
T_1&\ast&1&0\\
T_2&0&\ast &1\\
T_3&0&0&\ast
\end{array}
\end{array}
$$
$$
\begin{array}{cccc}
\begin{array}{c|c c c}
&T_1&T_2 &T_3\\
\hline
T_1&\ast&1&0\\
T_2&1&\ast &0\\
T_3&0&0&\ast
\end{array}
&
\begin{array}{c|c c c}
&T_1&T_2 &T_3\\
\hline
T_1&\ast&0&1\\
T_2&0&\ast &1\\
T_3&0&1&\ast
\end{array}
&
\begin{array}{c|c c c}

&T_1&T_2 &T_3\\
\hline
T_1&\ast&0&1\\
T_2&0&\ast &1\\
T_3&1&0&\ast
\end{array}
&
\begin{array}{c|c c c}
&T_1&T_2 &T_3\\
\hline
T_1&\ast&1&0\\
T_2&0&\ast &1\\
T_3&0&1&\ast
\end{array}
\end{array}
$$
$$
\begin{array}{cc}
\begin{array}{c|c c c}
&T_1&T_2 &T_3\\
\hline
T_1&\ast&1&1\\
T_2&1&\ast &0\\
T_3&1&0&\ast
\end{array}
&
\begin{array}{c|c c c}
&T_1&T_2 &T_3\\
\hline
T_1&\ast&1&1\\
T_2&1&\ast &1\\
T_3&1&1&\ast
\end{array}
\end{array}
$$
\normalsize
\end{example}

\begin{remark}
By looking at Example \ref{ex:HBC3}, one would expect to obtain a characterization for the Hilbert basis of the monoid $\mathscr{C}_n$ in a manner similar to the characterization for the Hilbert basis of the monoid $\mathscr{R}_n$ given by Theorem \ref{theo:HB}. For $n=3,4$ it is true that the set of extreme rays coincides with the Hilbert basis, while this is not longer true for $n \geq 5$.

In particular, for $n = 5$ the table
$$
\begin{array}{c|c c c c c}
&T_1&T_2 &T_3&T_4&T_5\\
\hline
T_1&\ast&7&5 &2&3\\
T_2&3&\ast &5&4&5\\
T_3&3&5&\ast &4&2\\
T_4&0&4&1&\ast &5\\
T_5&3&5&2&0&\ast
\end{array}
$$
is an element of the Hilbert basis, but it is not an extreme integral generator of the cone $\mathbb{R}_{+}\mathscr{C}_5$.  In fact, we have $\# \mathrm{HB}(\mathscr{C}_5)=4643$, while the cardinality of the set of extreme integral generators is $3441$.
\end{remark}

This leads to the formulation of the following natural question:

\begin{question}
Is there any general description that can be given for the Hilbert basis of the monoid $\mathscr{C}_n$?
\end{question}

\section{Multiplicity, Hilbert series and Applications} \label{sec:qua}

\subsection{The case $n=3$}

This subsection will be concerned with the explicit computation of two invariants of the affine monoid of both the consistent and the runner-up consistent score sheets of a round-robin football tournament played by $n=3$ teams, namely its multiplicity and Hilbert series, using the software Normaliz \cite{Nmz}. Potential applications of these computations are also presented.
\medskip

\begin{proposition}
Let $Q_r(G)$ denote the number of runner-up consistent score sheets with given number of goals $G$. Then, $Q_r(G)$ is a quasipolynomial
of degree $5$ and period $6$. Further, the Hilbert series $\sum_{G=0}^{\infty} Q_r(G)t^G$ is
$$
\frac{1+ 4t^2 +3t^3 +6t^4+ 5t^5+ 17t^6 -2t^7+ 19t^8+ 5t^9+ 8t^{10}+ 3t^{11}+ 8t^{12} -2t^{13}+ 3t^{14}}
{(1-t)^2(1-t^3)(1-t^6)^3}.$$ Thus, in terms of a quasipolynomial formula we have
$$
Q_r(G)=\left \{ \begin{array}{ll}
1+\frac{59}{45}G  +\frac{95}{144}G^2 + \frac{215}{1296} G^3 + \frac{107}{5184} G^4 + \frac{13}{12960} G^5& \mbox{~if~} G \equiv 0 \mod 6,\\
\\
\frac{713}{1728}+\frac{10957}{12960}G  +\frac{1453}{2592}G^2 + \frac{23}{144} G^3 + \frac{107}{5184} G^4 + \frac{13}{12960}  G^5& \mbox{~if~} G \equiv 1 \mod 6,\\
\\
\frac{7}{9}+\frac{887}{810}G  +\frac{775}{1296}G^2 + \frac{23}{144} G^3 + \frac{107}{5184} G^4 + \frac{13}{12960}  G^5& \mbox{~if~} G \equiv 2 \mod 6,\\
\\
\frac{43}{64}+\frac{1573}{1440}G  +\frac{181}{288}G^2 + \frac{215}{1296} G^3 + \frac{107}{5184} G^4 + \frac{13}{12960}  G^5& \mbox{~if~} G \equiv 3 \mod 6,\\
\\
\frac{20}{27}+\frac{431}{405}G  +\frac{767}{1296}G^2 + \frac{23}{144} G^3 + \frac{107}{5184} G^4 + \frac{13}{12960}  G^5& \mbox{~if~} G \equiv 4 \mod 6,\\
\\
\frac{259}{576}+\frac{11357}{12960}G  +\frac{1469}{2592}G^2 + \frac{23}{144} G^3 + \frac{107}{5184} G^4 + \frac{13}{12960}  G^5& \mbox{~if~} G \equiv 5 \mod 6.
\end{array}
\right.
$$
\end{proposition}

\begin{figure}[h]
\centering
\includegraphics[width=0.6\textwidth]{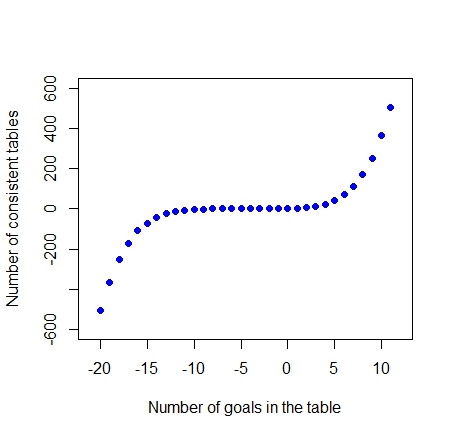}
\caption{Graph of $Q_c(G)$ }\label{fig:Q}
\end{figure}

\begin{proposition}
Let $Q_c(G)$ denote the number of consistent score sheets with given number of goals $G$. Then, $Q_c(G)$ is a quasipolynomial
of degree $5$ and period $12$. Furthermore, the Hilbert series $\sum_{G=0}^{\infty} Q_c(G)t^G$ is $R(t)/(1-t)^2(1-t^3)(1-t^6)^2(1-t^{12})$, where
\begin{align*}
R(t)&=1+ 3t^2 +1t^3 +6t^4+ 2t^5+ 10t^6 +t^7+ 13t^8+ 3t^9+ 11t^{10}+ \\
&+ 3t^{11}+ 13t^{12} +t^{13}+ 10t^{14}+2t^{15}+ 6t^{16}+t^{17}+ 3t^{18}+ t^{20}.
\end{align*}

Thus, in terms of a quasipolynomial formula we have

$$
Q_c(G)=\left \{ \begin{array}{ll}
1+\frac{49}{45}G  +\frac{1}{2}G^2 + \frac{299}{2592} G^3 + \frac{91}{6912} G^4 + \frac{91}{155520} G^5& \mbox{~if~} G \equiv 0 \mod 12,\\
\\
\frac{34291}{62208}+\frac{44263}{51840}G  +\frac{14521}{31104}G^2 + \frac{889}{7776} G^3 + \frac{91}{6912} G^4 + \frac{91}{155520}  G^5& \mbox{~if~} G \equiv 1 \mod 12,\\
\\
\frac{3245}{3888}+\frac{1679}{1620}G  +\frac{947}{1944}G^2 + \frac{889}{7776} G^3 + \frac{91}{6912} G^4 + \frac{91}{155520}  G^5& \mbox{~if~} G \equiv 2 \mod 12,\\
\\
\frac{145}{256}+\frac{5327}{5760}G  +\frac{185}{384}G^2 + \frac{229}{2592} G^3 + \frac{91}{6912} G^4 + \frac{91}{155520}  G^5& \mbox{~if~} G \equiv 3 \mod 12,\\
\\
\frac{214}{243}+\frac{1649}{1620}G  +\frac{943}{1944}G^2 + \frac{889}{7776} G^3 + \frac{91}{6912} G^4 + \frac{91}{155520}  G^5& \mbox{~if~} G \equiv 4 \mod 12,\\
\\
\frac{35315}{62208}+\frac{45223}{51840}G  +\frac{14585}{31104}G^2 + \frac{889}{7776} G^3 + \frac{91}{6912} G^4 + \frac{91}{155520} G^5& \mbox{~if~} G \equiv 5 \mod 12,\\
\\
\frac{15}{16}+\frac{49}{45}G  +\frac{1}{2}G^2 + \frac{299}{2592} G^3 + \frac{91}{6912} G^4 + \frac{91}{155520} G^5& \mbox{~if~} G \equiv 6 \mod 12,\\
\\
\frac{30403}{62208}+\frac{44263}{51840}G  +\frac{14521}{31104}G^2 + \frac{889}{7776} G^3 + \frac{91}{6912} G^4 + \frac{91}{155520}  G^5& \mbox{~if~} G \equiv 7 \mod 12,\\
\\
\frac{218}{243}+\frac{1679}{1620}G  +\frac{947}{1944}G^2 + \frac{889}{7776} G^3 + \frac{91}{6912} G^4 + \frac{91}{155520}  G^5& \mbox{~if~} G \equiv 8 \mod 12,\\
\\
\frac{161}{256}+\frac{5327}{5760}G  +\frac{185}{384}G^2 + \frac{299}{2592} G^3 + \frac{91}{6912} G^4 + \frac{91}{155520}  G^5& \mbox{~if~} G \equiv 9 \mod 12,\\
\\
\frac{3181}{3888}+\frac{1649}{1620}G  +\frac{943}{1944}G^2 + \frac{889}{7776} G^3 + \frac{91}{6912} G^4 + \frac{91}{155520} G^5& \mbox{~if~} G \equiv 10 \mod 12,\\
\\
\frac{31427}{62208}+\frac{45223}{51840}G  +\frac{14585}{31104}G^2 + \frac{889}{7776} G^3 + \frac{91}{6912} G^4 + \frac{91}{155520}  G^5& \mbox{~if~} G \equiv 11 \mod 12.
\end{array}
\right.
$$
\end{proposition}

Remark that $Q(-G)=-Q(G-9)$, as follows from Theorem \ref{cor:Gorenstein}. This identity may be observed in the graph of $Q$, which is presented in Figure \ref{fig:Q}.

\begin{definition}
A function $R:\ZZ\to\CC$ is called a \emph{quasi rational function} if is the quotient of two quasi-polynomials, see \cite{E}; or, in other words,
$$
R(n)=\frac{a_s(n)n^s+a_{s-1}(n)n^{s-1}+\cdots+a_1(n)n+a_{0}(n)}{b_t(n)n^t+b_{t-1}(n)n^{t-1}+\cdots+b_1(n)n+b_{0}(n)},
$$
where $a_i,b_j:\ZZ\to\CC$ are periodic functions for $i=0,\dots,s$, $j=0,\ldots , t$ and $a_s,b_t\neq 0$.
The degree of $R$ is the difference $s-t$, and the \emph{period} of $R$ is the smallest positive
 integer $p$ such that
 $$
 a_i(n+mp)=a_i(n)\  \ \mbox{and} \ \  b_j(n+mp)=b_j(n)
 $$
 for all $n,m\in \ZZ$ and $i=0,\dots,s$, $j=0,\ldots , t$.
\end{definition}

If $M, N$ are affine monoids with $N\subseteq M$ with Hilbert quasipolynomials $Q_M, Q_N$ of periods $m$ resp.~$n$, then it makes sense to construct the quasi rational function $R=Q_N/Q_M$. This has obviously period $\mathrm{lcm}(m,n)$. Applications are illustrated in the following.

\begin{remark}
Every value of $R$ may encode a certain conditional probability if we interpret the results in an appropriate way.
\end{remark}

\begin{example}
Let $Q_r(G)$ resp.~$Q(G)$ denote the number of runner-up consistent score sheets with given number of goals $G$ resp.~the number of ordered score sheets with given number of goals $G$, \cite[Proposition 5.4]{IM}. Observe that $Q_r(G)$ resp.~$Q(G)$ are quasipolynomials of degree $5$ and period $6$. Then $R_r(G)=Q_r(G)/Q(G)$ is a quasi rational function of period $6$ and degree $0$ given by
$$
R_r(G)=\left \{ \begin{array}{ll}
\frac{4320 + 4944 G + 2026G^2 + 379 G^3 + 26 G^4}{4320 + 5064 G + 2246 G^2 + 459 G^3 + 36G^4}& \mbox{~if~} G \equiv 0 \mod 6,\\
\\
\frac{2139 + 3955 G + 2115 G^2 + 405 G^3 + 26 G^4}{2079 + 3825 G + 2205 G^2 + 495 G^3 + 36 G^4}& \mbox{~if~} G \equiv 1 \mod 6,\\
\\
\frac{2520 + 1658 G + 379 G^2 + 26 G^3}{2160 + 1638 G + 459 G^2 + 36 G^3}& \mbox{~if~} G \equiv 2 \mod 6,\\
\\
\frac{5805 + 7503 G + 2929 G^2 + 457 G^3 + 26 G^4}{6345 + 7833 G+ 3299 G^2 + 567 G^3 + 36 G^4}& \mbox{~if~} G \equiv 3 \mod 6,\\
\\
\frac{1200 + 974 G + 275 G^2 + 26 G^3}{1080 + 954 G + 315 G^2 + 36 G^3}& \mbox{~if~} G \equiv 4 \mod 6,\\
\\
\frac{1665 + 1342 G + 327 G^2 + 26 G^3}{1485 + 1332 G + 387 G^2 + 36 G^3}& \mbox{~if~} G \equiv 5 \mod 6.
\end{array}
\right.
$$

Let $E_r (G)$ resp.~$E(G)$ be the event that $S$ is a runner-up consistent score sheet  with $G$ goals resp.~the event that $S$ is an ordered score sheet with $G$ goals. Then, for every $G \in \mathbb{N}$,
$$
R_r(G)=P(E_r(G) \mid E(G)),
$$
i.e., the \emph{conditional probability} of occurring $E_r(G)$ provided $E(G)$.
\end{example}

\begin{remark}\label{limit}
In Figure \ref{fig:pro} the dot-points are the graph of $R_r$ from the above example, while the crosses the graph of the analogous quasi rational function defined by $R_c=Q_c/Q$, or in other words, the conditional probability that the event that $S$ is a runner-up consistent score sheet  with $G$ goals occurrs provided that $S$ is an ordered score sheet. The blue dotted line equals the limit
$$
\lim_{G \to \infty} R_r(G)=\frac{e(\mathscr{R}_3)}{e(\mathscr{M}_3)}=\frac{\mathrm{vol}(\mathscr{P}_3)}{\mathrm{vol}(\mathscr{P})}=\frac{13}{18}\approx  0.722222,
$$
and the red dotted line equals the limit
$$
\lim_{G \to \infty} R_c(G)=\frac{e(\mathscr{C}_3)}{e(\mathscr{M}_3)}=\frac{\mathrm{vol}(\mathscr{D}_3)}{\mathrm{vol}(\mathscr{P})}=\frac{91}{216}\approx  0.421296.
$$
\end{remark}

\begin{figure}[h]
\centering
\includegraphics[width=0.5\textwidth]{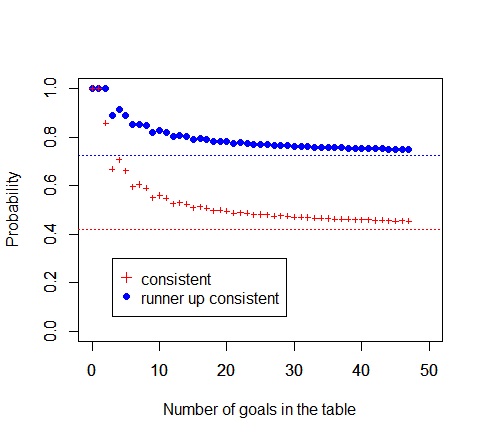}
\caption{Probabilities in the case $n=3$}\label{fig:pro}
\end{figure}

\subsection{Results obtained for $n\geq 4$}
We briefly present the results obtained for $n\geq 4$. The Hilbert series are computable up to the runner-up consistent $6\times 6$ and the consistent $5\times 5$ cases, however they are too complex for an explicit presentation here (especially their nominator). In order to show the increase in complexity we note that the denominator of the Hilbert series in the runner-up consistent $4\times 4$ case is
$$
(1-t)^2 (1-t^2) (1-t^4)^2(1-t^{12})^{7},
$$
and the denominator of the Hilbert series in the consistent $4\times 4$ case is
$$
(1-t)^2 (1-t^2) (1-t^4)^2(1-t^{12})^{3}(1-t^{120})(1-t^{360})(1-t^{2520})(1-t^{27720}).
$$

The multiplicities are computable up to the runner-up consistent $8\times 8$ and the consistent $5\times 5$ cases. Their exact writing as rational numbers involves very large integers and they are not so easy to interpret directly. For practical reasons we have included approximative results for the corresponding conditional probabilities in Table \ref{prob_data} (in which we use the notations from Remark \ref{limit}). We observe that the probabilities are decreasing fast as the size of the tournament increases.

\begin{table}[hbt]
 \centering
 \begin{tabular}{ccccc}

          & $\lim_{G \to \infty} R_r$    & $\lim_{G \to \infty} R_c$   \\
\midrule[1.2pt]
\ttt{3x3} &  0.722222 & 0.421296 \\ \hline
\ttt{4x4} &  0.512196 & 0.042183 \\ \hline
\ttt{5x5} &  0.351755 & 0.000648 \\ \hline
\ttt{6x6} &  0.235064 &    -     \\ \hline
\ttt{7x7} &  0.153372 &    -     \\ \hline
\ttt{8x8} &  0.097947 &    -     \\
\midrule[1.2pt]
\end{tabular}
\vspace*{2ex}\caption{Probabilities when $G \to \infty$}\label{prob_data}
\end{table}

\section{Computational experiments} \label{sec:comp}

The results presented in the previous sections were first conjectured by extensive computational experiments.
In this section we document these experiments, in the hope that the data may be useful for readers interested in running similar experiments.

For the experiments we have used the software Normaliz \cite{Nmz} (version 3.9.3), together with the the graphical interface jNormaliz \cite{AI}. For the algorithms implemented in Normaliz, we recommend the reader to see Bruns and Koch \cite{BK}, Bruns and Ichim \cite{BI}, \cite{BI2}, \cite{BI3}, Bruns, Hemmecke, Ichim, K\"oppe and S\"oger \cite{BHIKS}, Bruns and S\"oger \cite{BS}, Bruns, Ichim and S\"oger \cite{BIS}. All computations times in the tables below are ``wall clock times'' taken on a Dell R640 system with $1$ TB of RAM and two Intel\texttrademark Xeon\texttrademark Gold 6152 (a total of $44$ cores) using $20$ parallel threads (of the maximum of $88$).

\subsection{Hilbert bases computations using Normaliz}

We have run experiments using both the primal (see \cite{BI} and \cite{BHIKS}) and the dual algorithm of Normaliz (see \cite{BI}).
Table \ref{times1} contains the computation times for the
Hilbert bases of the corresponding monoids that we have obtained in these experiments.

\begin{table}[hbt]
\begin{tabular}{rrrrrr}
\strut                 & \multicolumn{2}{c}{runner-up consistent} & & \multicolumn{2}{c}{consistent}   \\
\cline{2-3}
\cline{5-6}
\strut                 &  \multicolumn{1}{c}{\ttt{primal}}&  \multicolumn{1}{c}{\ttt{dual}} & & \multicolumn{1}{c}{\ttt{primal}} &  \multicolumn{1}{c}{\ttt{dual}} \\
\midrule[1.2pt]
\strut \ttt{3x3}  &  0.013 s  &     0.028 s  & &     0.011 s &    0.044 s \\
\hline
\strut \ttt{4x4}  &  0.023 s  &     0.055 s  & &     0.030 s &    0.064 s \\
\hline
\strut \ttt{5x5}  &  0.213 s  &     0.079 s  & &   3:31:51 h &   26.145 s \\
\hline
\strut \ttt{6x6}  &  2.032 s  &     0.501 s  & &           - &  - \\
\hline
\strut \ttt{7x7}  &   2:06 m  &    12.688 s  & &           - &  - \\
\hline
\strut \ttt{8x8}  &  -        &      7:10 m  & &           - &  - \\
\midrule[1.2pt]
\end{tabular}
\vspace*{1ex}
\caption{Computation times for Hilbert bases}\label{times1}
\end{table}

We discuss the observations made during these experiments in the following remark.

\begin{remark} (a) As can be seen from the table, the dual algorithm is better suited for computing the Hilbert bases of these particular families of monoids.
This confirms the empirical conclusion that the dual algorithm is faster when the numbers of support hyperplanes is small relative to the dimension (see \cite{BIS} for more relevant examples).

(b) We also note that the Hilbert basis of the \ttt{9x9} runner-up consistent score sheets is very likely computable with Normaliz and the dual algorithm, however we were not interested
in this experiment since we were able to formulate Theorem \ref{theo:HB} without it.

(c) When using the primal algorithm for the runner-up consistent score sheets, a partial triangulation of the cone $C_n=\RR_+\mathscr{R}_n$
is builded by Normaliz (for details see \cite[Section 3.2]{BHIKS}). This partial triangulation is a subcomplex of the full lexicographic triangulation
obtained by inserting successively the extreme rays (for details see \cite[Section 4]{BI}).
For all the runner-up consistent score sheets for which computations with the primal algorithm were performed we observed that the partial triangulation is empty! Note that this is the
case if and only if the full lexicographic triangulation is a unimodular triangulation.
\end{remark}

\subsection{Multiplicities and Hilbert Series}
For computing multiplicities we have run experiments using the four algorithms that are now available in Normaliz:
\begin{enumerate}
\item The \emph{primal} Normaliz algorithm (see \cite{BI} and \cite{BIS}), which is denoted by \ttt{primal}. This algorithm is also available for computing Hilbert series;
\item \emph{Symmetrization} -- this is an improvement of the primal algorithm useful under certain conditions for computing both multiplicities and Hilbert series. It is presented in \cite{BS} and denoted by \ttt{symm};
\item \emph{Descent in the face lattice} -- denoted by \ttt{desc}. For a detailed discussion of this algorithm we refer the
reader to \cite{BI2};
\item \emph{Signed decomposition} -- denoted by \ttt{sgndec}. This a recent development of an algorithm first introduced by Lawrence \cite{Law} in the language of linear programming. More details are presented in \cite{BI3}.
\end{enumerate}

Table \ref{times2} contains the computation times for the
multiplicities of the corresponding monoids.

\begin{table}[hbt]
\begin{tabular}{rrrrrrrrr}
\strut            & \multicolumn{4}{c}{runner-up consistent} & & \multicolumn{3}{c}{consistent}   \\
\cline{2-5}
\cline{7-9}
\strut            &  \multicolumn{1}{c}{\ttt{primal}}&  \multicolumn{1}{c}{\ttt{symm}} & \multicolumn{1}{c}{\ttt{desc}} &  \multicolumn{1}{c}{\ttt{sgndec}} & & \multicolumn{1}{c}{\ttt{primal}} & \multicolumn{1}{c}{\ttt{desc}} & \multicolumn{1}{c}{\ttt{sgndec}} \\
\midrule[1.2pt]
\strut \ttt{3x3}  &  0.012 s &     0.012 s  &     0.012 s &    0.157 s  & &   0.011 s &  0.018 s &    0.156 s  \\
\hline
\strut \ttt{4x4}  &  0.033 s &     0.016 s  &     0.017 s &    0.159 s  & &   0.034 s &  0.025 s &    0.171 s \\
\hline
\strut \ttt{5x5}  &   1:19 m &     0.066 s  &     0.042 s &    0.181 s  & & 3:31:54 h & 41.205 s &     6:21 m \\
\hline
\strut \ttt{6x6}  &        - &    15.606 s  &     0.438 s &    1.094 s  & &         - &        - &          - \\
\hline
\strut \ttt{7x7}  &        - &           -  &    25.878 s &   41.733 s  & &         - &        - &          - \\
\hline
\strut \ttt{8x8}  &        - &           -  &   1:05:39 h &    41:00 m  & &         - &        - &          - \\
\midrule[1.2pt]
\end{tabular}
\vspace*{1ex}
\caption{Computation times for multiplicities}\label{times2}
\end{table}

Table \ref{times3} contains the computation times for the
Hilbert series.

\begin{table}[hbt]
\begin{tabular}{rrrrr}
\strut                 & \multicolumn{2}{c}{runner-up consistent} & & \multicolumn{1}{c}{consistent}   \\
\cline{2-3}
\cline{5-5}
\strut                 &  \multicolumn{1}{c}{\ttt{primal}}&  \multicolumn{1}{c}{\ttt{symm}} & & \multicolumn{1}{c}{\ttt{primal}}  \\
\midrule[1.2pt]
\strut \ttt{3x3}  &  0.013 s  &     0.015 s  & &     0.012 s  \\
\hline
\strut \ttt{4x4}  &  0.031 s  &     0.039 s  & &     1.221 s  \\
\hline
\strut \ttt{5x5}  &   4:18 m  &     0.907 s  & &  22:41:47 h  \\
\hline
\strut \ttt{6x6}  &  -        &      8:03 m  & &           -  \\
\midrule[1.2pt]
\end{tabular}
\vspace*{1ex}
\caption{Computation times for Hilbert series}\label{times3}
\end{table}

We discuss the observations made during these experiments in the following remark.

\begin{remark}\label{rem:tri} (a) As can be seen from the tables, the descent in the face lattice algorithm is clearly better suited for computing the multiplicities of the consistent score sheets,
while symmetrization should be used when computing the Hilbert series of the runner-up consistent score sheets.
Which algorithm should be used when computing the multiplicities of the runner-up consistent score sheets -- this is not as clear. Both the descent in the face lattice and the signed decomposition algorithms perform satisfactory and deliver quite similar performance.

(b) For computing the multiplicities (or Hilbert series) of the consistent score sheets symmetrization is not useful. Therefore results are not included above.

(c) When computing the multiplicity using the primal algorithm Normaliz is producing a full triangulation of the cone. In all the cases in which we were able to compute the multiplicity of the the runner-up consistent score sheets by this algorithm, the \emph{explicit} triangulation obtained was unimodular. As observed above, also the \emph{implicit} triangulations obtained in computations made for Hilbert bases were unimodular. We also note that, in general, the triangulations made by Normaliz are \emph{not} the
pulling triangulations implied by Corollary \ref{prop:compressed}. This is a remarkable fact, and one may ask if it is true for all triangulations of the runner-up consistent score sheets, or in other words: is the polytope $\mathscr{P}_n$ defined in the proof of Theorem \ref{theo:HB} totally unimodular? We also note that this is not true in the case of consistent score sheets, since in the case of the \ttt{4x4} consistent score sheets we have produced  triangulations which are \emph{not} unimodular.
\end{remark}

\end{document}